\documentclass[final]{elsarticle}

\usepackage{hyperref}
\usepackage{mathrsfs} 
\usepackage{amsmath}
\usepackage{amssymb}
\usepackage{vmargin}
\usepackage{theorem}

\newcommand{\e} {\varepsilon}
\newcommand{\C} {\mathbb{C}}
\newcommand{\D} {\mathbb{D}}

\newcommand{\N} {\mathbb{N}}

\newcommand{\R} {\mathbb{R}}

\newcommand{\GL}{\operatorname{GL}}

\newcommand{\p} {\mathbf{P}}

\newcommand{\NN}{\mathbb{N}}

\newcommand{\PP}{\mathbf{P}}

\newcommand{\ee}{{\bf \operatorname{e}}}

\newcommand{\FG}{\mathscr{F}_{n,m}(\p)}

\newtheorem{lemma}{Lemma}

\newtheorem{theorem}{Theorem}
 
\newtheorem{proposition}{Proposition} 
\newtheorem{remark}{Remark}
\newtheorem{example}{Example}
\newenvironment{proof}{\textbf{Proof.}}{\hfill $\Box$}

\journal{Journal of \LaTeX\ Templates}

\begin{document}

\begin{frontmatter}

\title{Feedback equivalence and uniform ensemble reachability}
\tnotetext[mytitlenote]{This work was supported by the German Research Foundation (DFG) under grant SCHO 1780/1-1.}


\author{Michael Sch\"onlein\corref{mycorrespondingauthor}}
\address{Institute for Mathematics, University of W\"urzburg, Germany}
\cortext[mycorrespondingauthor]{Corresponding author}

\ead{schoenlein@mathematik.uni-wuerzburg.de}




\begin{abstract}
This paper considers feedback methods for ensemble reachability of parameter-dependent linear systems $(A(\theta),B(\theta))$, where the parameter $\theta$ is varying over a compact Jordan arc in the complex plane. Recently, pointwise testable sufficient conditions for uniform ensemble reachability have been developed. Beside the necessity of pointwise reachablility these conditions put restrictions on the spectra of the matrices $A(\theta)$ and the Hermite indices of the pair $(A(\theta),B(\theta))$. In this paper we show that these conditions can be ensured by applying a suitable feedback transformation if the pair $(A(\theta),B(\theta))$ is pointwise reachable and it Kronecker indices are independent from the parameter. 
\end{abstract}


\end{frontmatter}


\section{Introduction}

An emerging field in mathematical systems and control theory  refers to the task of controlling a large, potentially infinite, number of states, or systems, using a single input function or a single feedback controller. Main goals of a control theory for such problems is to establish fundamental system theoretic methods and results in the context, i.e., to prove the existence of parameter independent open-loop and closed-loop controllers, develop methods for constructing them and tackle relevant system theoretic tasks.  This is a huge area and the term {\it ensemble control} has been established to refer to this {area of research, cf.~\cite[Section~2.4]{brockett2012notes}.

Ensembles arise in statistical approaches to linear systems, where the states are replaced by probability density functions. The design of controllers that morph one probability density function into another one then become control problems of the Fokker-Planck equation or the Liouville transport equations, cf. \cite{aborzi,brockett2012notes,chen2017optimal,fleig2016estimates,shen2017discrete,zeng_2016_observation,zeng2016ensemble}.

Ensemble control also embraces the situation of uncertainties in the model parameters. In this cases the task  is to control a parameter-dependent system with a single or a few open-loop inputs which are independent from the unknown 
model parameters \cite{li2009ensemble}. Recently, there has been much interest in motion control problems for spatio-temporal systems and infinite platoons of vehicles~\cite{bamieh2002distributed,curtain_tac_2015,melzer1971optimal,rogge2008vehicle}. Using Fourier-transform techniques, such control systems can be identified with parameter-dependent families of linear systems, cf.~\cite{bamieh2002distributed,Curtain2009}.

Besides, this topic is not entirely new and other terms then ensemble control are also present in the literature such as 
{\it simultaneous control} or {\it control of families of systems}, cf. \cite{buscain_2012,ghosh2000sufficient,hautussontag1986,loheac:hal-01164525,sontagwang1990}. Other closely related topics are robust control \cite{amato2006}, and the blending problem, as considered in \cite{Tannenbaum1981invariance}.
Sontag (together with Hautus and Wang) considered stabilization and pole-shifting for parameterized families of systems, cf. \cite{hautussontag1986,sontag_intro_families,sontagwang1990}. Besides, the series of papers by Ghosh \cite{ghosh1986_partial_pp,ghosh1985,ghosh1986_I,ghosh1988_II} investigates the possibility of simultaneously solving classical control problems for finitely many systems.

For recent contributions to the controllability problem for parameter-dependent systems we refer to  \cite{Agrachev_ensemble_lie_2016,agrachev2020control}, \cite{chen2019mcss}, \cite{li2011,li2009ensemble,li_tac_2016} and  \cite{Zeng_scl_2016}. Agrachev and Sarychev consider ensemble controllability for nonlinear drift-free parameter-dependent systems and provide a characterization in terms of Lie-brackets. In the same direction, the work of Chen \cite{chen2019mcss} also treats nonlinear systems and considers Lie extensions. We note that these approaches do not apply to the setting in this thesis. In \cite{li2011} a characterization for ensemble controllability for time-varying parameter-dependent linear systems is presented, which is based on the singular value decomposition of the reachability operator.

A recent observation from the control of probability densities is that the range of achievable tasks becomes much wider by using a mixture of open-loop and feedback controllers. In \cite{brockett2012notes} and \cite{DIRR20161018} it is shown that with pure open-loop controls only the mean value can be controlled, whereas by using additionally a feedback controller it is possible to control the mean value as well as the variance of the probability distribution. In the context of parameter-dependent systems and ensemble reachability, however, feedback methods have not been addressed so far. This paper devoted to this problem. In our analysis embraces the Hermite and the Kronecker indices. We note that in a series of papers Barag\~{a}na, Zaballa and co-workers tackled the relation between the Hermite indices and the controllability indices as well as the their behavior under perturbation and feedback, cf. \cite{Bar-Fer-Zab-00,Bar-Fer-Zab-05,Zab-97}.

The {\it organization} of the paper is as follows. In Section~\ref{sec:Not_defs} we introduce the class of systems under consideration. It also provides the definition of uniform ensemble reachability and recalls a known set of sufficient conditions that prepare the ground for the analysis in this paper. In this section we recall the definition of the Kronecker and Hermite indices and adapt  the notion of feedback equivalence to parameter-dependent systems. In Section~\ref{sec:restricted_equiv} we show that a pair of parameter-dependent matrices is restricted feedback equivalent to the Brunovsky from if its Kronecker indices do not depend on the parameter.
  %
%
%
%
Section~\ref{sec:main_results} contains the main results of the paper. That is, we show that for single-input pairs $(A,b)$ there is a continuous parameter-dependent feedback $f$ such that $(A+bf,b)$ is uniformly ensemble reachable if the pair is pointwise reachable. In this section we also show that in the for multi-input pairs $(A,B)$ there is feedback transformation so that the feedback transformed pair becomes uniformly ensemble reachable if it is pointwise reachable and the Kronecker indices are constant.
In Section~\ref{sec:open-feedback_osci} we investigate the controlled harmonic oscillator in light of the present context. For this example, we tackle the harder problem of deriving  a parameter-independent feedback. Moreover, under the additional assumption that the desired terminal states satisfy a Lipschitz condition, we also provide an estimate how the feedback gain influences the approximation of the terminal states.


\section{Problem statement, notation and known results}\label{sec:Not_defs}

In this paper we consider the reachability properties of parameter-dependent linear systems with the additional constraint that the open-loop control input is independent of the parameter. To investigate continuous-time systems 
\begin{align} \label{eq:cont-ensemble}
\tfrac { \partial } { \partial t } x ( t ,  \theta  ) &= A (  \theta ) x ( t ,  \theta  ) + B (  \theta  ) u ( t ) 
\end{align}
and discrete-time systems
\begin{align} \label{eq:disc-ensemble}
x _{ t + 1 }( \theta ) &= A ( \theta ) x _ t ( \theta ) + B ( \theta ) u_ t 
\end{align}
along the same lines the initial condition is in both cases $x(0,\theta)=x_0(\theta)=0$ for each parameter $\theta \in \p$. The parameter space $\p$ is assumed to be a Jordan arc in the complex plane, i.e. $\p$ is the image of a continuous and bijective function defined on a compact interval. The matrix-valued functions $A: \PP \mapsto \C^{n \times n}$ and $B: \PP \mapsto \C^{n \times m}$ are assumed to be continuous and we will use the short notation  $(A,B) \in C_{n,n}(\mathbf{ P })\times C_{n,m}(\mathbf{ P })$ to express this. Moreover, we denote by $C_n(\p)$ the space of continuous functions from $\p$ to $\C^n$. The time domain is $\NN_{0}$ in the discrete-time case or $[0,\infty)$ in the continuous-time case, and the inputs  are $u \in L^1_{\text{loc}}\left([0,\infty), \C^{m} \right)$ or $u = (u_{0},u_{1}, \dots), u_{i} \in \C^{m}$, respectively. Let 
\begin{align*}
\varphi ( T , 0, u ) (\theta) = \int _ { 0 } ^ { T } \mathrm { e } ^ { A ( \theta ) ( T - s ) } B ( \theta ) u ( s ) \mathrm { d } s
\end{align*}
 and
\begin{align*}
\varphi ( T , 0, u ) (\theta)  = \sum _ { k = 0 } ^ { T - 1 } A ( \theta ) ^ { k } B ( \theta ) u _ { T - 1 - k }
\end{align*}
denote the solutions of \eqref{eq:cont-ensemble} and \eqref{eq:disc-ensemble}, respectively.

The central notion of reachability that will be considered in this paper is as follows. A pair $(A,B)$  is called \emph{uniformly ensemble reachable (from zero)}, if for any $f \in C_n(\mathbf{ P })$ and any  $\varepsilon > 0$  there exist $T > 0$ and $u \in L^{1}_m ( [0,T] )$ or $ u = (u_{0}, u_{1}, \dots ,u_{T-1} ), \ u_{i} \in   \C^{m}$ such that 
\begin{align*}
\|\varphi (T, 0 , u)-f\|_{\infty}= \sup_{\theta \in \mathbf{ P }}	\|\varphi (T, 0 , u) (\theta)- f (\theta) \| < \varepsilon .
\end{align*}

We note that the notion ensemble reachability coincides with approximate reachability for the infinite-dimensional system, cf. \cite{JDE_ensembles_2021}. Also, we note that in continuous-time ensemble reachability is independent from the final time, i.e. if it holds for some $T>0$ is also holds for every $T>0$, cf. \cite{engelnagel,triggiani75}. Thus, for continuous-time systems ensemble reachability is equivalent to complete ensemble reachability, i.e. for every $x_0,f \in C_n(\p)$, for  every $\e>0$ and for every $T>0$  there is an input $u \in L^{1}_m ( [0,T] )$ such that $
\|\varphi (T, x_0 , u)-f\|_{\infty}<\e$.
Also we recall that exact ensemble reachability (i.e. $\e=0$) is never possible if the parameter space is infinite. In contrast, if $\p=\{ \theta_1,...,\theta_N\}$ is finite, ensemble reachability boils down to the classical finite-dimensional reachability of the corresponding parallel connection defined by the pair
 \begin{align*}
\begin{pmatrix} A(\theta_1)& & \\ & \ddots & \\ & & A(\theta_N) \end{pmatrix} , \quad \begin{pmatrix} B(\theta_1) \\ \vdots\\ B(\theta_N) \end{pmatrix} 
 \end{align*}

{\it Problem statement}: In this paper we explore how the application of feedback methods enlarges the class of the parameter-dependent linear systems that are uniformly ensemble reachable. We consider inputs of the form
\begin{align*}
u(t,x)=  F(\theta)x(t,\theta)+ u(t), \quad F\in C_{n,m}(\p).
\end{align*}
and aim at conditions on $(A,B) \in C_{n,n}(\p) \times C_{n,m}(\p)$ guaranteeing the existence of a feedback $F\in C_{n,m}(\p)$ and an open-loop input $u$ such that the mixed open-loop and feedback controlled pair $(A+BF,B)$ is uniformly ensemble reachable. 
 
%
%

In recent years, some effort has been spent to derive necessary and sufficient conditions for uniform ensemble reachability that are verifiable just in terms of the matrices $A(\theta)$ and $B(\theta)$, cf. \cite{JDE_ensembles_2021,li_tac_2016}. Exemplary, we recall the following set of sufficient conditions, cf. \cite[Corollary~4]{JDE_ensembles_2021}. 

In the case that $\mathbf{ P }$ is a Jordan arc a pair $(A,B)$ is uniformly ensemble reachable if the following conditions are satisfied:
\begin{enumerate}
\item[(N1)] $( A ( \theta ) , B ( \theta ) )$ is reachable for all $\theta \in \mathbf{ P }$.
\item[(N2)] For all distinct parameters $\theta, \theta' \in \mathbf{ P }$,  the spectra  $\sigma \big( A(\theta)\big)$ and $\sigma \big( A(\theta') \big)$ are disjoint.
\item[(S)] For each $\theta \in \mathbf{ P }$, the eigenvalues of $A(\theta )$ are simple.
\item[(H)] The Hermite indices $h_{1} (\theta), \dots , h_{m} (\theta) $ of $( A ( \theta ) , B ( \theta ) )$ are independent of $\theta \in \mathbf{ P }$.
\end{enumerate}

Before we recall the definition of the Hermite indices,  some comments are in order. Condition (N1) is also necessary for uniform ensemble reachability. Moreover, for single-input systems condition (N2) is also necessary for uniform ensemble reachability. 
The label for condition (S2) is chosen such that the notation is consistent with the labeling in \cite{MCSS_poly}.

Next we recall relevant lists of indices from finite-dimensional linear control theory, cf. \cite{kailath1980}. Let $(A,B) \in \mathbb{C}^{n\times n} \times \mathbb{C}^{n\times m}$ be a reachable pair, i.e.
\begin{align*}
 \operatorname{rank} \begin{pmatrix} B & AB  & \cdots & A^{n-1}B  \end{pmatrix} = n.
\end{align*}
Let $b_i$ denote the $i$th column of $B$. Selecting from left to right the first linear independent vectors
\begin{align}\label{eq:list-kronecker}
b _ { 1 } ,b_2,...,b_m,Ab_1,...,Ab_m, A^2b_1,...,A^2b_m, \ldots , A ^ {n - 1 } b _ { 1 } , \ldots , A ^ { n - 1 } b _ { m }
\end{align}
one obtains a list of basis vectors of the reachability subspace as
\begin{align*}
b _ { 1 } , \ldots , A ^ { \kappa _ { 1 } - 1 } b _ { 1 } , \ldots , b _ { m } , \ldots , A ^ { \kappa _ { m } - 1 } b _ { m }.
\end{align*}
The integers $\kappa(A,B)=(\kappa_{1},\dots,\kappa_{m})$ are called the \emph{Kronecker indices} of $(A,B)$, where $\kappa_{i} := 0$ if the vector $b_{i}$ has not been selected. 
Associated to the Kronecker indices $\kappa=(\kappa_1,...,\kappa_m)$ we define the pair $(A_{\kappa},B_{\kappa})$  given by
\begin{align}\label{eq:Brunovsky_restricted}
 A_{\kappa} =
 \begin{pmatrix}
 { A _ {\kappa_1}  } & { 0 } & { \cdots } & { 0 } \\ 
		{ 0 } & { A _ {  {\kappa_2 } } } & { \cdots } & { 0 } \\
		{ \vdots } & { \vdots } & { \ddots } & { \vdots }  \\
		{ 0 } & { 0 } & { \cdots } & { A _ { \kappa_{ m } } }
 \end{pmatrix}\quad \text{ and } \quad
 B _ {\kappa} = 
		\begin{pmatrix}
		{ b _ { \kappa_{ 1 } } } & { 0 } & { \cdots } & { 0 }  \\ 
		{ 0 } & { b _ { \kappa_{ 2 } } } & { \cdots } & { 0 }  \\
		{ \vdots } & { \vdots } & { \ddots } & { \vdots }  \\
		{ 0 } & { 0 } & { \cdots } & { b _ { \kappa_{ m } } }  
		\end{pmatrix},
\end{align}
where for $\kappa_i \geq 1$ the blocks $ A _ { \kappa_i }$ have size $k_ { i } \times k_ { i }$ and the $b _ { \kappa_i } $ are columns of size $k_ { i }$ and have the following form
\begin{align*}
A _ { \kappa_{ i } }= 
\left( 
\begin{array} { c c c c c } { 0 } & { 0 } & { 0 } & { \dots } & { 0 } \\ { 1 } & { 0 } & { 0 } & { \dots } & { 0 } \\ { \vdots } & { \ddots } & { \ddots } & { \ddots } & { \vdots } \\ { 0 } & { \ddots } & { 1 } & { 0 } & { 0 } \\ { 0 } & { 0 }  & { \ldots } & 1 & { 0 } \end{array} \right), \quad 
b_{k _ { i }} = 
\left( \begin{array} { c } { 1 } \\ { 0 } \\ { \vdots } \\ { 0 } \\ { 0 } 
\end{array} 
\right)
\end{align*}
and if $\kappa_i=0$ the block $A_{\kappa_i}$ is absent and $b_{\kappa_i}=0$.


Besides, selecting from left to right the first independent vectors
\begin{align*}
b _ { 1 } ,Ab_1,...,A^{n-1}b_1,b_2,Ab_2,...,A^{n-1}b_2,..., b_m,Ab_m, \ldots , A ^ {n - 1 } b _ {m } 
\end{align*}
one obtains another list of basis vectors of the reachability subspace 
\begin{align*}
b _ { 1 } , \ldots , A ^ { h _ { 1 } - 1 } b _ { 1 } , \ldots , b _ { m } , \ldots , A ^ { h _ { m } - 1 } b _ { m }.
\end{align*}
The integers ${h}(A,B)=(h_{1},\dots,h_{m})$ are called the \emph{Hermite indices}, where $h_{i} := 0$ if the vector $b_{i}$ has not been selected. 
Recall that a pair $(A,B)$ is reachable if and only if
\begin{align*}
 \sum_{j=1}^{m} \kappa_{j} = \sum_{j=1}^{m} h_{j}  = n.
\end{align*}

Furthermore, we briefly adapt the notion of feedback equivalence from finite-dimensional linear systems to the present context of parameter-dependent linear systems. For a more comprehensive exposition, we refer to \cite[Sections~6.1 and~6.3]{fuhrmann2015mathematics}, \cite[Section~5.2]{sontag}.


A triple $( T , F , S)   \in C_{n,n}(\p) \times C_{m,n}(\p) \times C_{m,m}(\p)$ is called a \emph{restricted feedback transformation} if it satisfies $ T(\theta) \in \GL_{ n }( \mathbb { C } )$ and $S(\theta) \in \mathbf { U }_ { m }$ for every $\theta \in \p$, i.e. $T(\theta)$ is invertible and $U(\theta)$ is upper triangular so that all diagonal entries are $1$. The term restricted refers to the requirement that $S(\theta) \in\mathbf { U }_ { m }$. In contrast, $( T , F , S)   \in C_{n,n}(\p) \times C_{m,n}(\p) \times C_{m,m}(\p)$ would be called a \emph{feedback transformation} if $ T(\theta) \in \GL_ { n }(\C^n)$ and $S(\theta) \in \GL_ { m }( \mathbb { C } )$ for every $\theta \in \p$.

In the following, for $M \in C_{n,m}(\p)$ and $N\in C_{m,p}(\p)$ we will  write $M\, N$ as a short notation for $M(\theta)N(\theta)$ for every $\theta \in \p$. Moreover, the set of a restricted feedback transformations defines the \emph{restricted feedback group} $\mathscr{F}_{n,m}(\p)$, where the composition of two elements $( T_1 , F_1 , S_1) $ and $( T_2 , F_2 , S_2)$  is given by
\begin{align*}
\left( T _ { 1 } , F _ { 1 } , S _ { 1 } \right) \circ \left( T _ { 2 } , F _ { 2 } , S _ { 2 } \right) = \left( T _ { 1 } T _ { 2 } , F _ { 1 } T _ { 2 } + S _ { 1 } F _ { 2 } , S _ { 1 } S _ { 2 } \right).
\end{align*} 
Further, the {neutral element} of the restricted feedback group $\FG$ is $ \left( I _ { n } , 0 _ { n , m } , I _ { m } \right)$
and the {inverse} of an element $(T,F,S)$ is given by $ \left( T ^ { - 1 } , - S ^ { - 1 } F T ^ { - 1 } , S ^ { - 1 } \right)$.
The restricted feedback group $\FG$ acts on a pair $(A,B) \in C_{n,n}(\p) \times C_{n,m}(\p)$ as follows
\begin{align*} 
( A , B )  \quad \mapsto_{(T,F,S)}  \quad  \left( T (  A  -  B S  ^ { - 1 } F ) T ^ { - 1 } , T B S ^ { - 1 } \right)
\end{align*}
For notational convenience we will write
\begin{align*} 
(T,F,S) \cdot ( A , B )   := \left( T (  A  -  B S  ^ { - 1 } F ) T ^ { - 1 } , T B S ^ { - 1 } \right).
\end{align*}

Two pairs $(A_{1},B_{1}),(A_{2},B_{2})$ in $ C_{n,n}(\p) \times C_{n,m}(\p)$ are called  \emph{restricted feedback equivalent on} $\p$, denoted by $\sim_{\p}$, if there exists $(T,F,S) \in \FG$ such that 
\begin{align*}
(A_2,B_2) = (T,F,S) \cdot (A_1,B_1).
\end{align*}
%
%
%
%
%
%
Note that $(A_1,B_1) \sim_{\p} (A_2,B_2)$  if and only if there is a restricted feedback transformation $(T,F,S)$ such that
\begin{align}\label{eq:restr_feedback_equiv_FH}
\begin{split}
T A_{1} - A_{2} T  &= B_{2}  F  \\
T B_1  &= B_2  S.  
 \end{split}
\end{align}
Also, we recall that the Kronecker indices $\kappa(A,B)$ are invariant under restricted feedback transformations, cf. \cite[Lemma~6.16]{fuhrmann2015mathematics}, i.e. for all 
$(T,F,S) \in \FG$ it holds
\begin{align*}
\kappa(A,B)(\theta) = \kappa\left( T (  A  +  B  F ) T ^ { - 1 } , T B S ^ { - 1 } \right)(\theta).
\end{align*}
For future use we recap the following well-known result.

\begin{lemma} \label{lem:fundlemma}
Let $T \in C_{n,n}(\p)$ and suppose that $T(\theta)=\big(t_1(\theta),\cdots,t_n(\theta)\big) \in \GL_ { n } ( \mathbb { C } )$ for every $\theta \in \p$.
\begin{enumerate}
 \item[(a)] Then the mapping $\theta \mapsto T(\theta)^{-1}$ is continuous.
 \item[(b)] Let $v\colon \p \to \mathbb{C}^n$ be continuous. Then, the coordinates $\alpha_{1}(\theta),\dots,\alpha_{n}(\theta)$ of $v(\theta)$ with respect to the basis $t_{1}(\theta),\dots, t_{n}(\theta)$ depend continuously on $\theta$. 
\end{enumerate}
\end{lemma}

%
%
%

\section{Feedback Equivalence  for parameter-dependent linear systems and  canonical forms}\label{sec:restricted_equiv}

In this section we show that every pointwise reachable pair $(A,B) \in C_{n,n}(\p) \times C_{n,m}(\p)$ with constant Kronecker indices is restricted feedback equivalent to $(A_\kappa,B_\kappa)$. This result will be used in the proof of Theorem~\ref{thm:multi_uer-feedback} and might be of independent interest. The proof follows the exposition in \cite{fuhrmann2015mathematics}, where the finite-dimensional case is treated. The main step in the subsequent proof will be to conclude that the constructed restricted feedback transformation in \cite[Proof of Theorem~6.18]{fuhrmann2015mathematics} is continuous in the parameter.

\begin{theorem}\label{thm:brunovsky-cont}
Suppose that $(A,B) \in C_{n,n}(\p)\times C_{n,m}(\p)$ has constant Kronecker indices $\kappa(A,B)(\theta)=(\kappa_1,...,\kappa_m)$ satisfying $\sum_{i=1}^m \kappa_i=n$, then $(A,B) \sim_{\p} (A_{\kappa},B_{\kappa})$. 
\end{theorem}
\begin{proof}
Let $\kappa(A,B) = (\kappa_{1}, \cdots, \kappa_{m})$ denote the constant Kronecker indices. We will construct a suitable feedback transformation $(T,F,S) \in \FG$ in four steps. 

\medskip
{\it Step~1}: First we apply a transformation of the form $(I,0,U)$. By definition and since $(A,B)$ is pointwise reachable, it holds that 
\begin{align*}
 b_1(\theta),...,A(\theta)^{\kappa_1-1}b(\theta)_1,b(\theta)_2,...,A(\theta)^{\kappa_2-1}b(\theta)_2,...,b(\theta)_m,...,A(\theta)^{\kappa_m-1}b(\theta)_m
\end{align*}
is a basis of $\mathbb{C}^n$ for each $\theta \in \p$. By construction, for every $i =1,...,m$ there are functions $\alpha_{ij} \colon \p \to \C$ and $\beta_{ijl} \colon \p \to \C$ such that
\begin{align}\label{eq:dependency}
{A(\theta)^ { \kappa_ { i }}  b(\theta)_ { i } } =  {\sum_ {j<i}^{} } \alpha_{ ji  }(\theta) A(\theta) ^ { \kappa_ { i }} b(\theta)_{ j }   
+ \sum _{l=1}^{ \kappa_{ i }}    A(\theta) ^ { l-1 }      \sum _ { j =1  } ^ {m  }   \beta_{ i j l }(\theta)  b(\theta) _ { j }
\end{align}
Since the functions $\theta \mapsto A(\theta)^l b(\theta)_i$ are continuous for every $l=0,1,2,...$ and $i=1,...,m$, by Lemma~\ref{lem:fundlemma}~(b) we have that $\alpha_{ ji  } \in C(\p)$ and $\beta_{ i j l} \in C(\p)$. Then, we define the continuous upper-triangular matrix
\begin{align*}
U(\theta) =
\begin{pmatrix}
1 & -\alpha_{12}(\theta) & -\alpha_{13}(\theta) & \cdots & -\alpha_{1m}(\theta) \\ 
0 & 1 			 & -\alpha_{23} (\theta) & \cdots & -\alpha_{2m}(\theta) \\ 
\vdots & 0 & 1 & \ddots & \vdots \\ 
& & \ddots & \ddots & -\alpha_{(m-1)\,m}(\theta)  \\ 
0 & & & 0 & 1 
\end{pmatrix}.  
\end{align*} 
Then, for $\tilde B(\theta):=B(\theta)U(\theta)$ the columns of $\tilde B$ and $B$ are related as follows
\begin{align*}
{ \tilde b (\theta) _  i } = b (\theta) _ { i } - \sum _ { j<i }^{} \alpha _ { ji }(\theta) b(\theta)  _ { j } , \quad i = 1 , \ldots , m,
\end{align*}
or equivalently
\begin{align*}
{  b (\theta) _  i } = \tilde b (\theta) _ { i } - \sum _ { j<i }^{} \tilde \alpha _ { ji }(\theta) \tilde b(\theta)  _ { j } ,
\end{align*}
where $\tilde \alpha_{ji}(\theta)$ denote the entries of $U^{-1}$. Note that $\tilde \alpha_{ji} \colon \p \to \C$ are continuous functions. For each $i=1,...,m$ it follows from \eqref{eq:dependency} that
\begin{align}\label{eq:dependency2}
\begin{split}
A(\theta) ^ { \kappa_ { i }}  {\tilde b(\theta) _ i } 
&= \sum _ { l = 1 } ^ { \kappa_ { i }}  A(\theta) ^ { l -1} \sum _ { j = 1 } ^ { m }  { \beta} _ { i j l }(\theta)  { b (\theta)_  j }\\
&= \sum _ { l = 1 } ^ { \kappa_ { i }}  A(\theta) ^ { l-1 } \sum _ { j = 1 } ^ { m }  { \beta} _ { i j l }(\theta)  \left( \tilde b(\theta)_j - \sum_{\mu<j}^{} \tilde \alpha_{\mu j} \tilde b(\theta)_\mu \right)\\
&= \sum _ { l = 1 } ^ { \kappa_ { i }}  A(\theta) ^ { l -1} \sum _ { j = 1 } ^ { m }  {\tilde \beta} _ { i j l }(\theta)  {\tilde b (\theta)_  j },
\end{split}
\end{align}
with $\tilde{\beta} _ { i j l } \in C(\p)$. To see that these functions are continuous, observe that the $\tilde{\beta} _ { i j l }$ are compositions of  the continuous functions $\tilde \alpha_{ji}$ and ${\beta} _ { i j l }$.

{\it Step~2}: In this step we will construct a continuous transformation $(T^{-1},0,I)$ so that the $0$- and $1$-entries are at the right places. To this end, for $i = 1,\dots,m$ and $l := 2, \dots \kappa_{i}$ we define continuous vectors
\begin{align*}
\mathrm{v}_{1i}(\theta) &:= {\tilde b(\theta)_i} \\
\mathrm{v}_{li}(\theta) &:= A(\theta)^{l-1}{\tilde b}(\theta)_i - \sum _ { \mu = 1 } ^ { l- 1 }  A(\theta) ^ { l - 1 - \mu }    \sum _ { j = 1 } ^ { m } \tilde{\beta} _ { i j (\kappa_i+1-\mu) } (\theta) \tilde b(\theta)_j
\end{align*}
and the transformation
\begin{align*}
T(\theta) := 
\begin{pmatrix}
  \mathrm{v} _ { 1 1}   (\theta)& \cdots & \mathrm{v} _ { \kappa _ { 1 }1 } (\theta)& \mathrm{v} _ { 12 }(\theta) &  \cdots & \mathrm{v} _ {  \kappa _ { 2 }2 } (\theta) & \cdots & \mathrm{v} _ { 1 m } (\theta)&  \cdots & \mathrm{v} _ { \kappa _ { m } m } (\theta) 
\end{pmatrix}.
\end{align*}
To see that $T(\theta)$ is continuously invertible, by Lemma~\ref{lem:fundlemma}~(b) it suffices to show that for every $\theta \in \p$ the columns of $T(\theta)$ define a basis of $\C^n$. Indeed, fix $\theta^* \in \p$ and let 
\begin{align*}
 \mathscr{X} := \operatorname{span} \{ \mathrm{v}_{li} (\theta^*) \, | \, i=1,...,m,\, l=1,...,\kappa_i\}.
\end{align*}
Since $(A(\theta^*), B(\theta^*))$ is reachable the claim follows by verifying that $\mathscr{X}$ is $A(\theta^*)$-invariant. 
To ease notation, we drop the dependence on $\theta^*$ and shortly write 
$$\gamma_{il}:= \sum _ { j = 1 } ^ { m }  \tilde{\beta} _ { i j l } \tilde{b} _ { j }= \sum _ { j = 1 } ^ { m }  \tilde{\beta} _ { i j l } \mathrm{v} _ {1 j } \in\mathscr{X}$$ for a moment.  Then, for $i=1,...,m$ and $l < \kappa_{i}$ we have 
\begin{align}\label{eq:AonV1}
\begin{split}
A \mathrm{ v } _ {l i  } &= 
A^{(l+1)-1}\tilde{b}_{i} - A \left(      A ^ { l - 2 }  \gamma_{i\kappa_i} +  A ^ { l - 3 }  \gamma_{i(\kappa_i-1)} + \cdots \gamma_{i(\kappa_i-l)} \right) -\gamma_{i(\kappa_i-(l+1))} +\gamma_{i(\kappa_i-(l+1))} \\
&= A^{(l+1)-1}\tilde{b}_{i} - \sum _ { \mu = 1 } ^ { (l+1)-1 }     A ^ { (l+1)-1- \mu }  \gamma_{i(\kappa_i+1-\mu)}  + \gamma_{i(\kappa_1-(l+1))} =\mathrm{ v } _ { ( l + 1 )i } + \gamma_{i(\kappa_1-(l+1))} \in \mathscr{X}. 
\end{split}
\end{align}
For $l = \kappa_{i}$ we use the same reasoning as above together with $\eqref{eq:dependency2}$ and obtain
\begin{equation}\label{eq:AonV2}
\begin{split}
A \mathrm{ v } _ { \kappa _ { i } i} &=  A^{\kappa_{i}}\tilde{b}_{i} -  A \left( A ^ { \kappa_{i} - 2 }  \gamma_{ik_i}
+  A ^ { \kappa_{i} - 3 }  \gamma_{i(k_i-1)} + \cdots + \gamma_{i2} \right) -  \gamma_{i1} + \gamma_{i1} = \gamma_{i1} \in \mathscr{X}.
\end{split}
\end{equation}

\noindent
{\it Step~3:} Structure of $T(\theta)^{-1}A(\theta)T(\theta)$. From \eqref{eq:AonV1} and \eqref{eq:AonV2} we get the following block structure
\begin{align*}
T(\theta)^{-1}A(\theta)T(\theta)=
\begin{pmatrix}
{ \tilde{A}(\theta) _ { 11 } } & { \tilde{A}(\theta) _ { 12 } } & { \cdots } & { \tilde{A}(\theta) _ { 1 m } } \\
{ \tilde{A}(\theta) _ { 21 } } & { \tilde{A}(\theta) _ { 22 } } & { \cdots } & { \tilde{A}(\theta) _ { 2 m } } \\ 
{ \vdots } & { \vdots } & { \ddots } & { \vdots } \\ 
{ \tilde{A}(\theta) _ { m 1 } } & { \tilde{A}(\theta) _ { m 2 } } & { \cdots } & { \tilde{A}(\theta) _ { mm } } 
\end{pmatrix},
\end{align*}
where the diagonal blocks have the form
\begin{align*}
\tilde A(\theta)_{ii}
=
\begin{pmatrix}
{ \tilde{\beta}_{ii1}(\theta) } & { \tilde{\beta}_{ii2}(\theta)  } & { \cdots } & { \cdots } & { \tilde{\beta}_{ii \kappa_i}(\theta) } \\
{ 1 } & { 0 } & { \cdots } & { \cdots } & {0} \\ 
{ 0 } & { 1 } & { \ddots } & {} & {  \vdots  } \\
{ \vdots } & { \ddots } & { \ddots } & { 0 } & {\vdots} \\
{ 0 } & { \cdots } & { 0 } & { 1 } & { 0 } \\ 
\end{pmatrix}\in \mathbb{C}^{\kappa_i \times \kappa_i}
\end{align*}
and the off-diagonal blocks have the form
\begin{align*}
 \tilde{A}(\theta)_{ij}
 =	
\begin{pmatrix}
 { \tilde{\beta}_{  ji 1}(\theta) } & { \tilde{\beta} _{ ji 2  } (\theta)} & { \cdots } & { \tilde \beta_{ j i  \kappa_i } (\theta)} \\
 { 0 } & { 0 }  & { \cdots } & { 0 } \\
 { \vdots } & { \vdots } & {  } & {\vdots } \\
 { 0 } & { 0 }  & { \dots } & { 0 } \\ 
\end{pmatrix}\in \mathbb{C}^{\kappa_j \times \kappa_i}
,\quad i \neq j.
\end{align*}
Let  $\ee_{k}$ denote the $k$-th standard basis vector in $\mathbb{C}^n$. It follows
\begin{align*}
T(\theta) \ee_k = 
\begin{cases}
\mathrm{v}_{k\,1}(\theta) & \text{ if } k \leq \kappa_1\\ 
\mathrm{v}_{(k-\kappa_1)\,2}(\theta) & \text{ if } \kappa_1 < k \leq \kappa_1 + \kappa_2 \\
\vdots & \\
\mathrm{v}_{(k-\kappa_1 - \cdots - \kappa_{m-1})\,m} (\theta)& \text{ if } \kappa_1 + \cdots +\kappa_{m-1} < k 
\end{cases}
\end{align*}
and equivalently, we have
\begin{align}\label{eq:v-e_inv_T}
T^{-1} \mathrm{v}_{li} (\theta)= 
\begin{cases}
\ee_l & \text{ if } i=1, \, l =1,...,\kappa_1\\ 
\ee_{\kappa_1+l} & \text{ if } i=2, \, l =1,...,\kappa_2\\ 
\vdots\\
\ee_{\kappa_1+ \cdots + \kappa_{m-1} +l}  & \text{ if } i=m\,, l=1,...,\kappa_m  .
\end{cases}
\end{align}
From \eqref{eq:v-e_inv_T} it follows that the transformation $(T^{-1},0,I)$ acts on the matrix $\tilde B$ as follows
\begin{align*}
T(\theta)^{-1}\tilde B(\theta) 
=
\begin{pmatrix}
  \ee_{1}& \ee_{\kappa_{1} + 1 }& \cdots & \ee_{\kappa_{1} + \dots + \kappa_{m-1}+1}
 \end{pmatrix}.
\end{align*}

{\it Step~4}: The final step is to transform  $T(\theta)^{-1} A(\theta)T(\theta)$ into the desired form $A_{\kappa}$. In doing so, the numbers $\tilde{\beta}_{ijl}(\theta)$ in the blocks of ${A}(\theta)_{ij}$ have to be eliminated. This will be achieved by applying the continuous feedback transformation $(I, F(\theta),I)$, where $F(\theta)$ is defined as the block matrix 
\begin{align*} 
F(\theta)= \begin{pmatrix}
            F_1(\theta) &\cdots&  F_m(\theta)
           \end{pmatrix}
\end{align*}
where
\begin{align*}
 F_i(\theta) = 
           \begin{pmatrix} 
           -\tilde{\beta}_{i11}(\theta) & \cdots & -\tilde{\beta}_{i1\kappa_i}(\theta)\\\
            \vdots &   & \vdots\\
            -\tilde{\beta}_{im1}(\theta) & \cdots & -\tilde{\beta}_{im\kappa_{i}}(\theta)\
           \end{pmatrix} \in \mathbb{C}^{m \times \kappa_i}. 
\end{align*}
Hence, the application of the restricted feedback transformation $(I,F,I)$ to the pair $(T^{-1}AT,B_{\kappa})$ leads to
\begin{align*}
T(\theta)^{-1}A(\theta) T(\theta)   - B_{\kappa}  F(\theta)  = A_{\kappa}.
\end{align*}
%
%

%
In summary, for the restricted feedback transformation $(T^{-1},  F, U^{-1})$ we have
\begin{align*} 
T(\theta) A(\theta)  -   A_{\kappa}T(\theta)^{-1}   &= B_{\kappa} F(\theta)  \\
T(\theta)^{-1} B(\theta)&= B_{\kappa}U(\theta)^{-1}
\end{align*}
and, by \eqref{eq:restr_feedback_equiv_FH}, the claim follows.

%

\end{proof}

\section{Main Results}\label{sec:main_results}

The main results of the paper explore the possibility to derive uniformly ensemble reachable systems by using a mixture of open-loop inputs and feedback controllers.  We begin with the single-input case. In this case there is only one Hermite index and it is equal to $n$ if condition~(N1) holds.  
Also, the conditions (N2) and (S2) put restrictions on the spectra of the matrices $A(\theta)$. Under the assumption that the pairs are reachable for every parameter, the spectra of the matrices can be assigned arbitrarily by the the Pole-Shifting Theorem, cf. \cite[Theorem~6.23]{fuhrmann2015mathematics}.  We obtain the following result.

\begin{theorem}\label{thm:Ackermann-ensemble}
Let $\p$ be a Jordan arc and assume that  $(A,b) \in C_{n,n}(\mathbf{ P }) \times C_{n}(\mathbf{ P })$ is pointwise reachable. Then, there is a continuous feedback $f \in C_{1,n}(\mathbf{ P })$ such that the feedback pair $(A + b f, b)$ is uniformly ensemble reachable.
\end{theorem}

\medskip

\begin{proof}
Since $\p$ is Jordan arc, there is a continuous and injective function $\gamma\colon [0,1] \to  \p$. Then, for $1 \leq l\leq k <n$ we define the injective mappings
\begin{align*}
\lambda_{l}(\theta) := e^{2\pi i\,\left(\gamma^{-1}(\theta) \tfrac{l-1}{k} + (1-\gamma^{-1}(\theta))\left(\tfrac{l}{k}-\tfrac{1}{k+1}\right)  \right)}   \in \partial \D. 
\end{align*}
and $k < l \leq n$ 
\begin{align*}
\lambda_{l}(\theta) := (l+1)- \gamma^{-1}(\theta) \in \R.
\end{align*} 
Consequently, for all $k\neq l$ we have
\begin{align*}
\lambda_{l}(\p) \cap \lambda_k(\p) = \emptyset.
\end{align*}
Next we define the family of monic polynomials $(p_\theta)_{\theta \in \p}$ by
\begin{align*}
p_\theta (z) := \prod_{ i = 1 }^{n} (z - \lambda_{i}(\theta) ).
\end{align*}
Then, by Ackermann's Formula \cite[Theorem~6.20]{fuhrmann2015mathematics}, the family 
\begin{align*}
f(\theta) := (0, \dots , 0,1)R(A(\theta),b(\theta))^{-1} p_{\theta}(A(\theta))
\end{align*}
of state-feedback, where $R(A(\theta),b(\theta)) = \begin{pmatrix}b(\theta) & A(\theta)b(\theta)& \cdots & A(\theta)^{n-1}b(\theta)\end{pmatrix}$ is the reachability matrix, satisfies 
\begin{align*}
\det (zI - A(\theta) + b(\theta)f(\theta)) = p_\theta(z).
\end{align*} 
Thus, the spectral conditions (N2) and (S2) are fulfilled and it remains to prove that $\theta \mapsto F(\theta)$ is continuous on $\p$. Since $A(\theta)$ is continuous and $p_{\theta}$ is a polynomial, $\theta \mapsto p_{\theta}(A(\theta))$ is also continuous. Furthermore the reachability matrix $R(A(\theta),b(\theta))$ is continuous and invertible for every $\theta \in \p$. So, by Lemma~\ref{lem:fundlemma}~(a) its inverse is also continuous. Hence, $F(\cdot)$ is continuous on $\p$. Finally we note that for single-input systems condition (N1) implies condition (H).  This shows the assertion.
\end{proof}

\medskip

To treat the multi-input case, we note that, it is well-known that the Hermite indices are not invariant under feedback, cf. \cite{Bar-Fer-Zab-05} and we have the following statement.

\begin{theorem}\label{thm:multi_uer-feedback}
Let $\p$ be a Jordan arc and assume that the pair $(A,B) \in C_{n,n}(\mathbf{ P }) \times C_{n,m}(\mathbf{ P })$ is pointwise reachable and the Kronecker indices are constant. Then, there exists a restricted feedback transformation $(T,F,S) \in \FG$ such that the pair
\begin{align*}
(\tilde{A}, \tilde{B}) = (T,F,S) \cdot (A,B)
\end{align*}
 is uniformly ensemble reachable.
\end{theorem}

\medskip

\begin{proof}
Let $\kappa(A,B) =(\kappa_1,...,\kappa_m)$ denote the constant Kronecker indices of the pair $(A,B)$. The proof is carried out in three steps. 

\medskip
{\it Step~1:~Constructing a pair $(\tilde A, \tilde B)$ satisfying (N1), (N2), (S2) and (H).} 
 We define the pair $(\tilde A, \tilde B)$ by
 \begin{align*}
			\tilde{A}(\theta) := \left( \begin{array} { c c c c c } { 0 } & { 0 } & { \ldots } & { 0 } &  {a_{0}(\theta)}\\
			{ 1 } & { 0 } & { \ldots } & { 0 } & { a_{1}(\theta) }  \\
			{ 0 } & { 1 } & { \dots } & { 0 } & { a_{2}(\theta) } \\
			{ \vdots } & { \vdots } & { \ddots } & { \vdots } & { \vdots } \\
			{ 0 } & { 0 } & { \ldots } & { 1 } & { a_{n-1}(\theta) }  \end{array} \right)
		\end{align*}
		and 
		\begin{align*}
			\tilde{B} := \begin{pmatrix} \ee_{1}& e_{1+\kappa_{1}}& \ee_{1+ \kappa_{1} + \kappa_{2} }& \cdots& \ee_{1+ \kappa_{1} + \dots + \kappa_{m} } &0& \cdots& 0 \end{pmatrix}.
		\end{align*}
If $\kappa_{1} = n$, we define $\tilde{B} := \begin{pmatrix} \ee_{1} & 0 & \cdots& 0 \end{pmatrix}$. 
 
To show that the pair $( \tilde A, \tilde B)$ satisfies the conditions (N1), (N2), (S) and (H), observe that for $i=1,\dots,n-1$ one has
$$\tilde{A}(\theta)^{i}\ee_{1} = \ee_{i+1}\qquad \forall \, \, \theta \in \p.$$ Thus, the pair $( \tilde A, \tilde B)$ has constant Hermite indices $h(\tilde A,\tilde B)=(n,0,\dots,0)$. This shows that conditions~(N1) and~(H) are satisfied. As the functions $a_0,...,a_{n-1}$ are the coefficients of the characteristic polynomial of $\tilde A$, choosing them as in the proof of Theorem~\ref{thm:Ackermann-ensemble} the spectral conditions~(N2) and~(S2) are satisfied and the pair $(\tilde{A}, \tilde{B})$ is uniformly ensemble reachable.

\medskip

{\it Step~2:~$(\tilde A,\tilde B)$ has the same Kronecker indices as $(A,B)$.} 

We exemplary treat the first Kronecker index. The others follow from the same reasoning.  Note that for all $i=1,2,..., n-1$ it holds
\begin{align*}
\tilde A(\theta)^i\tilde b_1= \tilde A(\theta)^i\ee_1 = \ee_{i+1}
\end{align*}
and 
\begin{align*}
\tilde A(\theta)^i\tilde b_l= \tilde A(\theta)^i\ee_{1+ \kappa_1+ \cdots + \kappa_{l-1}} = \ee_{1+ \kappa_1+ \cdots + \kappa_{l-1}+i} 
\end{align*}
for all $ l=2,...,m-1$  and $ i =1,...,n-(1+\kappa_1+\cdots + \kappa_{l-1})$. Thus, for all $i<\kappa_1$ the vectors $\tilde A(\theta)^i\tilde b_1$ are linear independent from the vectors $\tilde b_1,....,\tilde b_m$ and $\tilde A(\theta)^i\tilde b_1,...,A(\theta)^i\tilde b_m $. For $i=\kappa_1$ one has $\tilde A(\theta)^{\kappa_1}\ee_1=\ee_{\kappa_1+1}=\tilde b_2$, which has already been selected. Thus, the first Kronecker index $k_1(\tilde A,\tilde B)=\kappa_1$.

\medskip

{\it Step~3:~Application of Theorem~\ref{thm:brunovsky-cont}.}  By  Theorem~\ref{thm:brunovsky-cont} there are feedback transformations $(T,F,S)$ and $(\tilde{T},\tilde{F},\tilde{S})$ such that
\begin{align*}
(T,F,S) \cdot (A,B )  
=(A_{k},B_{k}) =(\tilde{T},\tilde{F},\tilde{S}) \cdot (\tilde{A}, \tilde{B}).
\end{align*}
Thus, it holds
\begin{align*}
(\tilde{A}, \tilde{B}) = \left( (\tilde{T},\tilde{F},\tilde{S})^{-1} \circ (T,F,S)\right) \cdot (A,B).
\end{align*}
This shows the assertion.
\end{proof}

\medskip
We note that another well-known list is given by the controllability indices, cf. \cite{sontag} and \cite{hpII}. 
In \cite[p.~301]{fuhrmann2015mathematics}\footnote{In \cite{fuhrmann2015mathematics} the controllability are called reachability indices.} it is pointed out that if all Kronecker indices are non-zero, the controllability indices are obtained from the Kronecker indices by reordering them in decreasing form. Therefore, since the Kronecker indices in Example~\ref{ex:Hermite-kronecker} are non-zero, the example also shows that constant Hermite indices is independent from constant controllability indices and vice versa.

\begin{remark}
Following the proof of Theorem~14 in \cite{sontag}, Theorem~\ref{thm:brunovsky-cont} can easily be modified to show that for every  pair $(A,B) \in C_{n,n}(\p)\times C_{n,m}(\p)$ with constant controllability indices that sum up to $n$ there is a feedback transformation $(T,F,S)$ such that $(T,F,S)\cdot (A,B)$ is uniformly ensemble reachable.
\end{remark}

The following Example~\ref{ex:Hermite-kronecker}~(a) is taken from \cite{hpII} and can be used to show that the constancy of the Kronecker indices is independent from the constancy of the Hermite indices and controllability indices.  

\begin{example}\label{ex:Hermite-kronecker}
Let $\p=[-1,1]$. 
\begin{enumerate}
\item[(a)]
Consider the matrix pair $(A_1,B_1)$ defined by 
\begin{align*}
A_1(\theta) = 
\begin{pmatrix}
0 & 1 & 0 & 0\\ 
2\theta^2 & 0 & 0 & 2 \theta\\ 
0 & 0 & 0 & 1\\ 
0 & -2 \theta & 0 & 0\\ 
\end{pmatrix}
\qquad B_1(\theta)=  \begin{pmatrix}
0 & 0\\ 
1 & 0 \\ 
0 & 0 \\ 
0 & 1\\ 
\end{pmatrix}.
\end{align*}
The columns of the Kalman matrix are
\begin{align*}
\begin{pmatrix}
0 \\ 
1\\ 
0 \\ 
0 \\ 
\end{pmatrix},
\begin{pmatrix}
0\\ 
0 \\ 
0\\ 
1\\ 
\end{pmatrix},
\begin{pmatrix}
1\\ 
0\\ 
0 \\ 
-2\theta \\ 
\end{pmatrix},
\begin{pmatrix}
0\\ 
2\theta \\ 
1\\ 
0\\ 
\end{pmatrix},
\begin{pmatrix}
0 \\ 
0\\ 
-2\theta \\ 
0 \\ 
\end{pmatrix},
\begin{pmatrix}
2\theta\\ 
0 \\ 
0 \\ 
-4\theta^2\\ 
\end{pmatrix},
\begin{pmatrix}
0 \\ 
0\\ 
0 \\ 
0 \\ 
\end{pmatrix},
\begin{pmatrix}
0\\ 
-4\theta^3 \\ 
-4\theta^2\\ 
0\\ 
\end{pmatrix}.
\end{align*}
Hence, the pair $(A_1,B_1)$ has constant Kronecker indices $\kappa(A_1,B_1)\equiv(2,2)$ and the Hermite indices are given by
  \begin{align*}
   h(A_1,B_1)(\theta)= \begin{cases}
                (3,1) & \text{ if } \theta \neq 0\\
                (2,2) & \text{ if } \theta = 0.\\
               \end{cases}
  \end{align*}

  \item[(b)]  Consider the pair $(A_2,B_2)$ defined by 
  \begin{align*}
A_2(\theta) =
\begin{pmatrix}
0 & 0 & 2 & \theta^2-\tfrac{1}{2}\\ 
1 & 0 & 0 & 1\\ 
0 & 1 & 0 & 0\\ 
0 & 0 & 0 & 0\\ 
\end{pmatrix}
\qquad B_2(\theta)=  \begin{pmatrix}
0 & 0\\ 
1 & 0 \\ 
0 & 0 \\ 
0 & 1\\ 
\end{pmatrix}.
\end{align*}
The columns of the Kalman matrix are
\begin{align*}
\begin{pmatrix}
0 \\ 
1\\ 
0 \\ 
0 \\ 
\end{pmatrix},
\begin{pmatrix}
0\\ 
0 \\ 
0\\ 
1\\ 
\end{pmatrix},
\begin{pmatrix}
0\\ 
0\\ 
1 \\ 
0 \\ 
\end{pmatrix},
\begin{pmatrix}
\theta^2-\tfrac{1}{2}\\ 
1 \\ 
0\\ 
0\\ 
\end{pmatrix},
\begin{pmatrix}
2 \\ 
0\\ 
0 \\ 
0 \\ 
\end{pmatrix},
\begin{pmatrix}
0\\ 
\theta^2-\tfrac{1}{2} \\ 
1 \\ 
0\\ 
\end{pmatrix},
\begin{pmatrix}
0 \\ 
2\\ 
0 \\ 
0 \\ 
\end{pmatrix},
\begin{pmatrix}
2\\ 
0 \\ 
\theta^2-\tfrac{1}{2} \\ 
0\\ 
\end{pmatrix}.
\end{align*}
Hence, the pair $(A_2,B_2)$ has constant Hermite indices $h(A_2,B_2)\equiv(3,1)$ and the Kronecker indices are given by
  \begin{align*}
   \kappa(A_2,B_2)(\theta)= \begin{cases}
                 (3,1) & \text{ if } \theta^2 \neq \tfrac{1}{2}\\
                 (2,2) & \text{ if } \theta^2 = \tfrac{1}{2}.\\
                 \end{cases}
  \end{align*}
 \end{enumerate}
\end{example}


\section{Open-Loop and Feedback Controlled Harmonic Oscillators}\label{sec:open-feedback_osci}

In this section, we consider an ensemble of controlled harmonic oscillators and investigate the possibility to use a mixture of an open-loop and constant feedback controller of the form 
$$u(t,y)= ky(t,\theta)+ u(t),  \quad k \in \R.$$
For notational convenience we denote the feedback gain by $k \in \mathbb{R}$ instead of $F$ as in the previous section.
Then, let $g(\theta)$ denote the input function associated with the parameter $\theta\in\p:= [-\theta^*,\theta^*]\subset \R$. The dynamic equation reads as follows
\begin{align}\label{eq:harm_osci-mixed}
 \tfrac{\partial^2}{\partial t^2} y(t,\theta) + \theta^2 y(t,\theta) = g(\theta) \,\left( ky(t,\theta) + u(t) \right) .
\end{align}
In order to establish conditions guaranteeing the existence of a $k \in \R$ such that \eqref{eq:harm_osci-mixed} is uniformly ensemble reachable we consider the corresponding first order system
\begin{align}\label{eq:harm_osci-1st-order-mixed}
\tfrac{\partial}{\partial t} x(t,\theta) =     A_k(\theta) x(t,\theta) + b_g(\theta) u(t)
\end{align}
with
\begin{align}
A_k(\theta):=
  \begin{pmatrix} 0 & 1\\ kg(\theta)-\theta^2  &0 \end{pmatrix}
 ,\quad   b_g(\theta) :=
  \begin{pmatrix} 0 \\ g(\theta) \end{pmatrix}.
\end{align}
Note that for $\p=[0,\theta^*]$ it follows from \cite[Theorem~4]{JDE_ensembles_2021} that the family of controlled harmonic oscillators \eqref{eq:harm_osci-1st-order-mixed} is uniformly ensemble reachable by means of pure open-loop controller, i.e. $k=0$. The following result states conditions such that the application of $u(t,y)=ky(t,\theta)+ u(t) $ yields uniform ensemble reachability over the parameter space $[-\theta^*,\theta^*]$. Our first result is as follows.

\begin{proposition}\label{prop:ex-k}
Let $\theta^*>0$ and $\p=[-\theta^*,\theta^*]$ and suppose that $g \in C^1(\p)$ is zero-free and strictly monotone. Then, for $k^* :=\max_{\theta \in \p}  \tfrac{2\theta}{g'(\theta)}$ the open-loop and feedback controlled harmonic oscillators \eqref{eq:harm_osci-mixed} are uniformly ensemble reachable for all $k > k^*$ .
\end{proposition}
\begin{proof}
First, note that the Kalman matrix for \eqref{eq:harm_osci-1st-order-mixed} is given by $$g(\theta) \begin{pmatrix} 0 & 1\\ 1  &0 \end{pmatrix}.$$ Since $g(\theta)\neq 0$ for all $\theta \in [-\theta^*,\theta^*]$, the Kalman matrix has rank $2$, i.e. \eqref{eq:harm_osci-1st-order-mixed} is reachable for every $\theta$. The characteristic polynomial of $A_k(\theta)$ is given by 
\begin{align*}
z^2 - (kg(\theta)- \theta^2).
\end{align*}
By \cite[Theorem~4]{JDE_ensembles_2021} it suffices to show that there is a $k^* \in \R$ such that the
functions $h_k \in C^1(\p)$, $h_k(\theta):=kg(\theta)- \theta^2$ are injective for all $k> k^*$. To this end, let $k^* := \max_{\theta \in \p}  \tfrac{2\theta}{g'(\theta)} $. Then, for all $k> k^*$ one has $ h'_k(\theta):=kg'(\theta)- 2\theta >0$ and, thus, the functions $h_k$ are injective on $\p$ for all $k > k^*$. This shows the assertion.
%
%
%
%
%
%
\end{proof}

\medskip
We note that in the proof above it is also sufficient to pick $k^*$ such that $h'_k(\theta)<0$ for all $\theta \in \p$ and for all $k <k^*$. An appropriate choice in this case would be $k^*:= \min_{\theta \in \p}  \tfrac{2\theta}{g'(\theta)}$. Depending on the particular situation at hand it might be suitable to use the latter. Subsequently we investigate how to get an error bound and the influence of the feedback gain $k$ on it. Before we do so, we recap useful properties of Lipschitz continuous functions.

\begin{lemma}\label{lem:Lipschitz}
 Let $I$ and $J$ be compact intervals and suppose that $f\colon I \to \R$ and $g\colon J \to \R$ satisfy a Lipschitz condition with $L_f>0$ and $L_g>0$, respectively.
 \begin{enumerate}
 \item[(i)] If $g(J)\subset I$, then the composition $f\circ g$ satisfies a Lipschitz condition with $L_f L_g>0$.
  \item[(ii)] If $I=J$, then the product $fg \colon I \to \R$ satisfies a Lipschitz condition with $L_f M_g + L_g M_f>0$, where $M_f:=\max_{x\in I}|f(x)|$ and $M_g:=\max_{x\in I}|g(x)|$.
  \item[(iii)] If $f$ is zero-free, then $\tfrac{1}{f}$ satisfies a Lipschitz condition with $\tfrac{L_f}{m_f^2}$, where $m_f=\min_{x\in I}|f(x)|$.
  \item[(iv)] If $f$ is continuously differentiable and strictly monotone, then the inverse $f^{-1}$ satisfies a Lipschitz condition with $L_{f^{-1}} = (\min_{x\in I}|f'(x)|)^{-1}$.
 \end{enumerate}
\end{lemma}
%
%

In order to formulate the following result for continuous-time and discrete-time controlled harmonic oscillators, we state the error in terms of $\|p(A)b-f\|_\infty$ and note that in the continuous-time case, the subsequent estimate for $\|p(A)b-f\|_\infty$ has to be combined with an estimate for $\|\varphi(T,u,0)-p(A)b\|_\infty$, cf. \cite[Section~4]{MCSS_poly}.

\begin{proposition}\label{prop:osci-feedback-k-estimate}
Suppose that the assumptions of Proposition~\ref{prop:ex-k} hold  and suppose that $f \in C_2(\p)$ satisfies a Lipschitz condition. Let 
\begin{align*}
 k^*:= \max \left\{ \tfrac{\theta^*}{\min_{\theta \in \p} |g(\theta)| } ,\, \max_{\theta \in \p}  \tfrac{2 \theta}{g'(\theta)} \right\} .
\end{align*}
Then, for every $k >k^*$ there is a sequence of polynomials $(p_n)_n$ of the degree of $2n+1$  with $n\geq 3$ such that  for $c_{g,\p}:=\left|  g(\theta^*) - g(-\theta^*) \, \right|$ one has
\begin{align*}
\| p_n(A_k)b_g  - f\|_\infty \leq 
\tfrac{1}{g(-\theta^*)}\,\left(\tfrac{4 M_f}{\min_{\theta \in \p} |g(\theta)|}  +\,\tfrac{k\,c_{g,\p}}{2\min_{\theta \in \p} |kg'(\theta)-2\theta|}\left( M_{f} \,\tfrac{L_g}{m_g^2} + \tfrac{ L_f}{m_g}  \right) \right)\sqrt{\frac{\log n }{n}}.
\end{align*}
\end{proposition}
\begin{proof}
Since $g \in C^1(\p)$ is non zero and strictly monotone we assume  w.l.o.g. that $g(\theta)>0$ for every $\theta \in \p$. Then, for every $k >k^*$  it holds
\begin{align*}
 h_k(\theta) := kg(\theta)-\theta^2 > 0 \qquad \text{ and } \qquad h_k'(\theta) = kg'(\theta)-2 \theta >0 
\end{align*}
for all $\theta \in \p$. Then, the continuous transformation
\begin{align}
T(\theta)^{-1}:=
 \tfrac{1}{g(\theta)} \begin{pmatrix} 0 & 1\\ 1 &0 \end{pmatrix}
 \end{align}
yields that 
\begin{align*}
T(\theta)^{-1}A_k(\theta) T(\theta) =  \begin{pmatrix} 0 & h_k(\theta)\\  1 &0 \end{pmatrix}, \quad   
T(\theta)^{-1}\begin{pmatrix} 0 \\ g(\theta) \end{pmatrix}= \begin{pmatrix} 1 \\ 0 \end{pmatrix}.
\end{align*}
For every polynomial $p$ we get
\begin{align*}
\| p(A_k)b_g  - f\|_\infty \leq \frac{1}{\min_{\theta \in \p} g(\theta)}  \, \left\| p \left( \begin{pmatrix} 0 & h_k\\ 1 &0 \end{pmatrix} \right) \begin{pmatrix} 1 \\ 0 \end{pmatrix}   -   \tfrac{1}{g} \begin{pmatrix} f_2 \\ f_1 \end{pmatrix} \right\|_\infty.
\end{align*}
As in \cite[Proof of Theorem~4]{JDE_ensembles_2021}, we consider a sequence of polynomials $(p_n)_{n \in \N}$ of the form
\begin{align*}
 p_n(z):=  q_n\big(z^2\big)+ r_n\big(z^2\big)z,
\end{align*}
for some polynomials $q_n$ and $r_n$ and obtain
\begin{align*}
 p_n \left( \begin{pmatrix} 0 & h_k(\theta)\\ 1 &0 \end{pmatrix} \right) \begin{pmatrix} 1 \\ 0 \end{pmatrix}   = \begin{pmatrix} q_n(h_k(\theta)) \\ r_n(h_k(\theta)) \end{pmatrix}.
\end{align*}
Hence, for $z \in h_k(\p)$ we have to investigate the terms
\begin{align*}
| q_n(z) - \tfrac{f_2}{g}\circ h_k^{-1}(z)|
\quad \text{and} \quad
| r_n(z) - \tfrac{f_1}{g}\circ h_k^{-1}(z)|.
\end{align*}
By Lemma~\ref{lem:Lipschitz}, the functions $\tilde f_1\colon h_k(\p) \to \R$ and $\tilde f_2\colon h_k(\p) \to \R$ defined by
\begin{align*}
\tilde f_1(z) :=  \frac{f_2(h_k^{-1}(z)) }{g(h_k^{-1}(z))} \quad \text{ and } \quad \tilde f_2(z) :=  \frac{f_1(h_k^{-1}(z)) }{g(h_k^{-1}(z))},
\end{align*}
satisfy a Lipschitz condition with
\begin{align*}
L_{\tilde f_1} = L_{h_k^{-1}} \left( M_{f_2} \,\frac{L_g}{m_g^2} + \frac{ L_f}{m_g}  \right)     \quad \text{ and } \quad  L_{\tilde f_2} = L_{h_k^{-1}} \left( M_{f_1} \,\frac{L_g}{m_g^2} + \frac{ L_f}{m_g}  \right)      ,
\end{align*}
respectively. By Lemma~\ref{lem:Lipschitz}~(iv) one has $L_{h_k^{-1}} =  \frac{1}{\min_{\theta \in \p} |kg'(\theta)-2\theta| }$. Then, we take $q_n$ and $r_n$ as the $n$th Bernstein polynomials to the functions $\tilde f_1$ and $\tilde f_2$, respectively, i.e.
\begin{align*}
q_n(z) := B_{n,\tilde f_2}(z) \quad \text{ and } \quad r_n(z):=B_{n,\tilde f_1}(z). 
\end{align*}
Then, by \cite[Theorem~1]{gzyl1997} we have
\begin{multline*}
\| p_n(A_k)b_g  - f\|_\infty \leq \tfrac{1}{g(-\theta^*)} \, \left\| p_n\left(\begin{pmatrix} 0 & h_k(\theta)\\ 1  &0 \end{pmatrix} \right)\binom{1}{0}  - \tfrac{1}{g}\binom{f_2}{f_1}\right\|_\infty \\
\leq \tfrac{1}{g(-\theta^*)} \, \max_{i=1,2} | B_{n,\tilde f_i}(z) - \tilde f_i(z)|_\infty 
\leq \tfrac{1}{g(-\theta^*)}\,\max_{i=1,2} \, \left\{  4 M_{\tilde f_i} + \frac{k\,c_{g,\p}L_{\tilde f_i}}{2}   \right\}\sqrt{\frac{\log n }{n}}\\
\leq \tfrac{1}{g(-\theta^*)}\,\left(\frac{4 M_f}{\min_{\theta \in \p} |g(\theta)|}  +\,\frac{k\,c_{g,\p}}{2\min_{\theta \in \p} |kg'(\theta)-2\theta|}\left( M_{f} \,\frac{L_g}{m_g^2} + \frac{ L_f}{m_g}  \right) \right)\sqrt{\frac{\log n }{n}}
\end{multline*}
for $n\geq 3$. This shows the assertion.
\end{proof}

\medskip

The latter error bound shows that the approximation is getting better the larger the feedback gain is. In terms of the eigenvalues of the matrices $A_k(\theta)$ it can be observed that the eigenvalue arcs $\theta \mapsto \sqrt{k g(\theta)-\theta^2}$ and $\theta \mapsto -\sqrt{k g(\theta)-\theta^2}$ are located on the positive and the negative real line, respectively, and the gap between them is increasing with the feedback gain $k$.  Also, we emphasize that for pure open-loop inputs, i.e. $k=0$, the pair $(A_0,b_g)$ is not uniformly ensemble reachable over the parameter set $\p=[-\theta^*,\theta^*]$. This corresponds to the fact that the traces of the eigenvalue arcs $\theta \mapsto i \theta$ and $\theta \mapsto -i\theta$ coincide. By using a feedback gain $k\in \R$ the traces can be separated and the assumptions in \cite[Theorem~3.1.1~(c)]{JDE_ensembles_2021} are satisfied such that the pair $(A_k,b_g)$ is uniformly ensemble reachability.

%
%
%

%
%
%
%



\end{document}